\documentclass{article}
\usepackage{amsthm}
\usepackage{amssymb}
\usepackage{amsmath}
\usepackage{epsfig}
\usepackage{colonequals}
\usepackage{enumerate}
\usepackage{hyperref}

\newcommand{\CC}{{\cal C}}

\newtheorem{theorem}{Theorem}[section]
\newtheorem{corollary}[theorem]{Corollary}
\newtheorem{lemma}[theorem]{Lemma}
\newtheorem{observation}[theorem]{Observation}

\begin{document}
\title{Fine structure of $4$-critical triangle-free graphs III.  General surfaces}
\author{%
     Zden\v{e}k Dvo\v{r}\'ak\thanks{Computer Science Institute (CSI) of Charles University,
           Malostransk{\'e} n{\'a}m{\v e}st{\'\i} 25, 118 00 Prague, 
           Czech Republic. E-mail: \protect\href{mailto:rakdver@iuuk.mff.cuni.cz}{\protect\nolinkurl{rakdver@iuuk.mff.cuni.cz}}.
           Supported by project 14-19503S (Graph coloring and structure) of Czech Science Foundation.}
\and	   
Bernard Lidick\'y\thanks{
Iowa State University, Ames IA, USA. E-mail:
\protect\href{mailto:lidicky@iasate.edu}{\protect\nolinkurl{lidicky@iastate.edu}}.
Supported by NSF grants DMS-1266016 and DMS-1600390.}
}
\date{\today}
\maketitle
\begin{abstract}
Dvo\v{r}\'ak, Kr\'al' and Thomas~\cite{trfree4,trfree6} gave a description of the structure of triangle-free graphs
on surfaces with respect to $3$-coloring.  Their description however contains two substructures
(both related to graphs embedded in plane with two precolored cycles) whose coloring properties
are not entirely determined.  In this paper, we fill these gaps.
\end{abstract}

\section{Introduction}

The interest in the $3$-coloring properties of planar graphs was started by a celebrated theorem of Gr\"otzsch~\cite{grotzsch1959},
who proved that every planar triangle-free graph is $3$-colorable.  This result was later generalized and strengthened
in many different ways.  The one relevant to the topic of this paper concerns graphs embedded in surfaces.
While the direct analogue of Gr\"otzsch's theorem is false for any surface other than the sphere, $3$-colorability of
triangle-free graphs embedded in a fixed surface is nowadays quite well understood.
Building upon several previous results (\cite{trfree3,gimbel,locplanq,MohSey,NakNegOta,tw-klein,thom-torus,thomassen-surf}),
Dvo\v{r}\'ak et al.~\cite{trfree4,trfree6} gave a description of the structure of triangle-free graphs
on surfaces with respect to $3$-coloring.  To state the description precisely, we need to introduce a number of definitions.

A \emph{surface} is a compact connected $2$-manifold with (possibly null) boundary.  Each component of the boundary
is homeomorphic to a circle, and we call it a \emph{cuff}.
Consider a graph $G$ embedded in the surface $\Sigma$; when useful, we identify $G$ with the topological
space consisting of the points corresponding to the vertices of $G$ and the simple curves corresponding
to the edges of $G$.  A \emph{face} $f$ of $G$ is a maximal connected subset of $\Sigma-G$.
A face $f$ is a \emph{closed $2$-cell} if it is homeomorphic to an open disk and its boundary forms a cycle $C$ in $G$;
the \emph{length} of $f$ is defined as $|E(C)|$.
A graph $H$ is a \emph{quadrangulation} of a surface $\Sigma$ if all faces of $H$ are closed $2$-cells and have length $4$
(in particular, the boundary of $\Sigma$ is formed by a set of pairwise vertex-disjoint cycles in $H$, called the \emph{boundary cycles} of $H$).
A vertex of $G$ contained in the boundary of $\Sigma$ is called a \emph{boundary vertex}.

Let $G$ be a graph, let $\psi$ be a proper coloring of $G$ by colors $\{1,2,3\}$ and let $Q=v_1v_2\ldots v_kv_1$ be a directed closed walk in $G$.
One can view $\psi$ as mapping $Q$ to a closed walk in a triangle $T$ with vertices $1$, $2$, and $3$, and the \emph{winding number} of $\psi$ on $Q$ is then the number of times this walk goes
around $T$ in a fixed orientation. More precisely, for $uv\in E(G)$, let $\delta_\psi(u,v)=1$ if $\psi(v)-\psi(u)\in\{1,-2\}$, and
$\delta_\psi(u,v)=-1$ otherwise.  For a walk $W=u_1u_2\ldots u_m$, let $\delta_\psi(W)=\sum_{i=1}^{m-1} \delta_\psi(u_i,u_{i+1})$.
The winding number $\omega_\psi(Q)$ of $\psi$ on $Q$ is defined as $\delta_\psi(Q)/3$.

Suppose that $G$ is embedded in an orientable surface $\Sigma$ so that every face of $G$ is closed $2$-cell.  Let $\CC$ be the set consisting of all facial
and boundary cycles of $G$.  For each of the cycles in $\CC$, choose an orientation of its edges with outdegree one
so that every edge $e\in G$ is oriented in opposite directions in the two cycles of $\CC$ containing $e$.  We call such orientations \emph{consistent}.
Note that there are exactly two consistent orientations opposite to each other.  There is a well known constraint on winding numbers in embedded graphs.

\begin{observation}\label{obs-win0}
Let $G$ be a graph embedded in an orientable surface $\Sigma$ so that every face of $G$ is closed $2$-cell.
Let $\CC$ be the set consisting of all facial and boundary cycles of $G$.
Let $Q_1$, \ldots, $Q_m$ be the cycles of $\CC$ viewed as closed walks in a consistent orientation.  If $\psi$ is a $3$-coloring of $G$, then
$$\sum_{i=1}^m \omega_\psi(Q_i)=0.$$
\end{observation}

As the winding number of any $3$-coloring of a $4$-cycle is $0$, we obtain the following constraint on $3$-colorings of quadrangulations.

\begin{corollary}\label{cor-win0}
Let $G$ be a quadrangulation of an orientable surface $\Sigma$.  Let $B_1$, \ldots, $B_k$ be the boundary cycles of $G$ in a consistent orientation.  If $\psi$ is
a $3$-coloring of $G$, then
$$\sum_{i=1}^k \omega_\psi(B_i)=0.$$
\end{corollary}

For non-orientable surfaces, the situation is a bit more complicated.  Suppose that $G$ is a quadrangulation of
a surface, and let us fix directed closed walks $B_1$, \ldots, $B_k$
tracing the cuffs of $G$.  For each facial cycle, choose an orientation arbitrarily.
Let $D$ denote the directed graph with vertex set $V(G)$ and $uv$ being an edge of $D$ if and only if $uv$ is an edge of $G$ oriented
towards $v$ in both cycles of $\CC$ that contain it.  Let $p(G,B_1,\ldots, B_k)=2|E(D)|\bmod 4$.
Note that $p(G,B_1,\ldots, B_k)$ is independent on the choice of the orientations of the $4$-faces, since reversing an orientation
of a $4$-face with $d$ edges belonging to $D$ changes $2|E(D)|$ by $2(4-2d)\equiv 0\pmod{4}$.

Consider the sum of winding numbers of a $3$-coloring $\psi$ of $G$ on cycles in $\CC$.
As before, the contributions of all edges of $G$ that do not belong to $D$ cancel out, and
since $G$ is a quadrangulation, the winding number on any non-boundary cycle in $\CC$ is $0$.
Hence,
$$\sum_{i=1}^k \delta_\psi(B_i)=2\sum_{uv\in E(D)} \delta_\psi(u,v).$$
Since $\delta_\psi(u,v)=\pm 1$ for every $uv\in E(D)$,
$$2\sum_{uv\in E(D)} \delta_\psi(u,v)\equiv 2|E(D)|\equiv -2|E(D)|\equiv -p(G,B_1,\ldots, B_k)\pmod{4},$$
regardless of the $3$-coloring $\psi$.  Furthermore,
$$\sum_{i=1}^k \delta_\psi(B_i)=3\sum_{i=1}^k \omega_\psi(B_i)\equiv -\sum_{i=1}^k \omega_\psi(B_i)\pmod{4}.$$
Therefore, we get the following necessary condition for the existence of a $3$-coloring.

\begin{observation}\label{obs-gen}
Let $G$ be a quadrangulation of a surface $\Sigma$.  Let $B_1$, \ldots, $B_k$ be the boundary cycles of $G$.  If $\psi$ is
a $3$-coloring of $G$, then
$$\sum_{i=1}^k \omega_\psi(B_i)\equiv p(G,B_1,\ldots, B_k)\pmod{4}.$$
\end{observation}

If a $3$-coloring $\psi$ of the boundary cycles satisfies the condition of Observation~\ref{obs-gen}, we say that $\psi$ is
\emph{parity-compliant}.

We say that a coloring $\psi$ of the boundary cycles of a quadrangulation of a surface $\Sigma$ \emph{satisfies the winding number constraint}
if either
\begin{itemize}
\item $\Sigma$ is orientable and the sum of winding numbers of $\psi$ on the boundary cycles of $G$ in their consistent orientation is $0$, or
\item $\Sigma$ is non-orientable and $\psi$ is parity-compliant.
\end{itemize}

The structure theorem of Dvo\v{r}\'ak, Kr\'al' and Thomas~\cite{trfree6} now can be stated as follows.

\begin{theorem}[\cite{trfree6}]\label{thm-struct}
For every surface $\Sigma$ and integer $k\ge 0$, there exists a constant $\beta_0$ with the following property.  Let $G$ be a triangle-free graph embedded in $\Sigma$
so that every cuff of $\Sigma$ traces a cycle in $G$ and so that the sum of the lengths of the boundary
cycles of $G$ is at most $k$.  Suppose that every contractible $4$-cycle in $G$ bounds a face.  Then $G$ has a subgraph $H$ with at most $\beta_0$
vertices, such that $H$ contains all the boundary cycles and each face $h$ of $H$ satisfies one of the following (where $G[h]$ is the subgraph
of $G$ drawn in the closure of $h$).
\begin{enumerate}[(a)]
\item Every precoloring of the boundary of $h$ extends to a $3$-coloring of $G[h]$, or
\item $G[h]$ is a quadrangulation and every precoloring of the boundary of $h$ which satisfies the winding number constraint
extends to a $3$-coloring of $G[h]$, or
\item $h$ is an open cylinder and $G[h]$ is its quadrangulation, or
\item $h$ is an open cylinder and both boundary cycles of $h$ have length exactly $4$.
\end{enumerate}
\end{theorem}

Given the subgraph $H$ as in the theorem, we can decide whether a given precoloring of the boundary cycles of $G$ extends to a $3$-coloring of $G$,
by first trying all possible (constantly many) $3$-colorings of $H$ that extend the precoloring, and then testing whether they extend to the subgraphs $G[h]$ for
each face $h$ of $H$.  In the cases (a) and (b), the theorem also shows how to test the existence of the extension to $G[h]$.  In the cases (c) and (d),
much fewer details are given (although the structure is still sufficiently restrictive to enable many applications of the theorem).  The goal of this paper
is to fill in this gap.  Before stating our result, let us first give a few more definitions.

We construct a sequence of graphs $T_1$, $T_2$, \ldots, which we call \emph{Thomas-Walls graphs} (Thomas and Walls~\cite{tw-klein} proved that they are exactly
the $4$-critical graphs that can be drawn on Klein bottle without contractible cycles of length at most $4$).
Let $T_1$ be equal to $K_4$.  For $n\ge 1$, let $uv$ be any edge of $T_n$ that belongs to two triangles and let $T_{n+1}$ be obtained from $T_n-uv$
by adding vertices $x$, $y$ and $z$ and edges $ux$, $xy$, $xz$, $vy$, $vz$ and $yz$.  The first few graphs of this sequence are drawn in Figure~\ref{fig-thomaswalls}. 
The sequence is unique but embeddings in the plane are not combinatorially unique. 
\begin{figure}
\begin{center}
\includegraphics[scale=0.8]{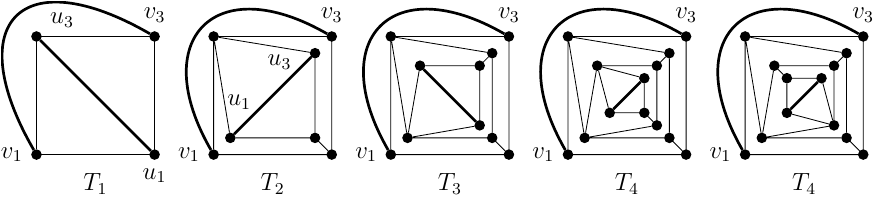}
\end{center}
\caption{Some Thomas-Walls graphs (with two different drawings of $T_4$).}\label{fig-thomaswalls}
\end{figure}
For $n\ge 2$, note that $T_n$ contains unique $4$-cycles $C_1=u_1u_2u_3u_4$ and $C_2=v_1v_2v_3v_4$ such that $u_1u_3,v_1v_3\in E(G)$.
Let $T'_n=T_n-\{u_1u_3,v_1v_3\}$. 
We call the graphs $T'_1$, $T'_2$, \ldots \emph{reduced Thomas-Walls graphs}, and we say that $u_1u_3$ and $v_1v_3$ are their \emph{interface pairs}.
Note that for $n\ge 3$, the $4$-cycles $C_1$ and $C_2$ are vertex-disjoint, and thus $T'_n$ has an embedding in the cylinder with boundary cycles $C_1$ and $C_2$ (as we mentioned before, this embedding is not unique for $n\ge 4$).
We always use only such embeddings with boundary $4$-cycles when drawing $T_i$ in the cylinder.

A \emph{patch} is a graph drawn in the disk with a cycle $C$ of length $6$ tracing the cuff, such that $C$ has no chords and every face of the patch other than the one bounded by $C$ has length $4$.
Let $G$ be a graph embedded in a surface.
Let $G'$ be any graph which can be obtained from $G$ as follows.
Let $S$ be an independent set in $G$ such that every vertex of $S$ has degree $3$.  For each vertex $v\in S$
with neighbors $x$, $y$ and $z$, remove $v$, add new vertices $a$, $b$ and $c$ (drawn very close to the original position of $v$) and $6$-cycle
$C=xaybzc$, and draw any patch with the boundary cycle $C$ in the disk bounded by $C$.
We say that any such graph $G'$ is obtained from $G$ by \emph{patching}.
This operation was introduced by Borodin et al.~\cite{4c4t} in the context of describing planar $4$-critical graphs with exactly $4$ triangles.

Consider a reduced Thomas-Walls graph $G=T'_n$ for some $n\ge 1$, with interface pairs $u_1u_3$ and $v_1v_3$.
A \emph{patched Thomas-Walls graph} is any graph obtained from such a graph $G$ by patching, and $u_1u_3$ and $v_1v_3$ are
its interface pairs.

Next, let us construct another class of graphs, \emph{forced extension quadrangulations}.  Let $G$ be a graph
embedded in the cylinder whose holes are contained in distinct faces of $G$ bounded by (not necessarily disjoint)
cycles $C_1$ and $C_2$ of the same length $k$ (let us remark that these cycles are disjoint from the cuffs).
We say that $G$ is a forced extension quadrangulation if all other
faces of $G$ have length $4$, every non-contractible cycle in $G$ has length at least $k$, and
there is a sequence $K_1$, \ldots, $K_n$ of non-contractible $k$-cycles in $G$ such that $C_1=K_1$, $C_2=K_n$
and $K_i$ intersects $K_{i+1}$ in at least one vertex for $1\le i\le n-1$.
Note that if $C_1$ and $C_2$ are vertex-disjoint, then $G$ can be drawn as a quadrangulation of the cylinder with boundary cycles $C_1$ and $C_2$.
If the distance between $C_1$ and $C_2$ is at least $4k$, we say that $G$ is a \emph{wide forced extension quadrangulation}.

Our main result now can be stated as follows.
\begin{theorem}\label{thm-structmain}
For every surface $\Sigma$ and integer $k\ge 0$, there exists a constant $\beta$ with the following property.  Let $G$ be a triangle-free graph embedded in $\Sigma$
so that every cuff of $\Sigma$ traces a cycle in $G$ and so that the sum of the lengths of the boundary
cycles of $G$ is at most $k$.  Suppose that every contractible $(\le\!5)$-cycle in $G$ bounds a face.  Then $G$ has a subgraph $H$ with at most $\beta$
vertices, such that $H$ contains all the boundary cycles and each face $h$ of $H$ satisfies one of the following (where $G[h]$ is the subgraph
of $G$ drawn in the closure of $h$).
\begin{enumerate}[(a)]
\item Every precoloring of the boundary of $h$ extends to a $3$-coloring of $G[h]$, or
\item $G[h]$ is a quadrangulation and every precoloring of the boundary of $h$ which satisfies the winding number constraint
extends to a $3$-coloring of $G[h]$, or
\item $h$ is an open cylinder and $G[h]$ is its wide forced extension quadrangulation with boundary cycles whose length is divisible by $3$, or
\item $h$ is an open cylinder and $G[h]$ is a patched Thomas-Walls graph.
\end{enumerate}
\end{theorem}

Let us remark that Aksenov~\cite{aksenov} showed that in a triangle-free planar graph $G'$, every precoloring of a $(\le\!5)$-cycle extends to a $3$-coloring of $G'$.  Hence, given an embedded graph,
removing vertices and edges contained inside a contractible $(\le\!5)$-cycle does not affect the $3$-coloring properties of the graph.  Therefore, the assumption of Theorem~\ref{thm-structmain}
that $G$ contains no non-facial contractible $(\le\!5)$-cycles does not affect its generality.

Which precolorings of the boundary cycles of a patched Thomas-Walls graph $F$ extend to a $3$-coloring of $F$ was determined in
the first paper of this series, see Lemma 2.7 of~\cite{cylgen-part1}.  In Section~\ref{sec-feq}, we describe which precolorings
of the boundary cycles of a wide forced extension quadrangulation $F$ extend to a $3$-coloring of $F$, thus making
Theorem~\ref{thm-structmain} fully explicit about the $3$-coloring properties of $G$.  In Section~\ref{sec-cyl}, we prove
Theorem~\ref{thm-structmain} for quadrangulations of the cylinder.  Finally, Section~\ref{sec-main} is devoted
to the proof of Theorem~\ref{thm-structmain}.

\section{Forced extension quadrangulations}\label{sec-feq}

A function $\theta:V(G)\to V(H)$ is a \emph{homomorphism} if $\theta(u)\theta(v)\in E(H)$ for every edge $uv\in E(G)$.
If $C'$ and $C$ are cycles of the same length, then a homomorphism $\theta:V(C')\to V(C)$ is \emph{cyclic} if
it is an isomorphism.  If $\theta':V(C')\to V(C)$ is another cyclic homomorphism, we say that $\theta'$ and $\theta$
are \emph{offset by $2$} if the distance between $\theta(v)$ and $\theta'(v)$ in $C$ is $2$ for every $v\in V(C')$.
If $X \subseteq V(G)$ and $\theta:V(G)\to V(H)$, then $\theta\restriction X$ denotes the restriction of a $\theta$ to $X$.

\begin{lemma}\label{lemma-hom}
Let $G$ be a graph embedded in the cylinder so that the holes of the cylinder are contained in distinct faces bounded
by cycles $C_1$ and $C_2$ of the same length $k$,
such that all other faces of $G$ have length $4$.  Suppose that every non-contractible cycle in $G$ has
length at least $k$.  Let $C$ be a $k$-cycle.  Then there exists a homomorphism $\theta:V(G)\to V(C)$ such that $\theta$ restricted to $C_1$ as well as $C_2$
is cyclic.  Furthermore, if $G$ is not a forced extension quadrangulation, then there exists another such homomorphism $\theta'$ such that
$\theta'\restriction V(C_1)=\theta\restriction V(C_1)$, and $\theta'\restriction V(C_2)$ and $\theta\restriction V(C_2)$ are offset by $2$.
\end{lemma}
\begin{proof}
We prove the lemma by induction on the number of vertices of $G$.

Suppose that $G$ contains a non-contractible $k$-cycle $K$ distinct from $C_1$ and $C_2$.
For $i\in\{1,2\}$, let $G_i$ be the subgraph of $G$ drawn between $C_i$ and $K$.  Let $\theta_i:V(G_i)\to V(C)$
be the homomorphism obtained by the induction hypothesis, chosen so that $\theta_1\restriction V(K)=\theta_2\restriction V(K)$.
Then, $\theta_1\cup \theta_2$ gives a homomorphism from $G$ to $C$ as required.
If $G$ is not a forced extension quadrangulation, then by symmetry, we can assume that $G_2$ is not a forced extension quadrangulation,
and thus there exists a homomorphism $\theta'_2:V(G_2)\to V(C)$ such that $\theta_2\restriction V(K)=\theta'_2\restriction V(K)$,
and $\theta_2\restriction V(C_2)$ and $\theta'_2\restriction V(C_2)$ are offset by $2$.
Hence, $\theta_1\cup \theta'_2$ is the other homomorphism required by the claim of Lemma~\ref{lemma-hom}.

Therefore, we can assume that every non-contractible cycle distinct from $C_1$ and $C_2$
has length greater than $k$.  Note that each such cycle has the same parity as $C_1$, and thus it has
length at least $k+2$.

Suppose that $G$ contains a contractible cycle $K=z_1z_2z_3z_4$ of length 4 which does not bound a face.
Let $G'$ be obtained from $G$ by removing the interior of $K$. 
By the induction hypothesis, $G'$ has a homomorphism $\theta$ to $C$.
Since $G$ is a quadrangulation, the subgraph of $G$ induced by $K$ and its interior is bipartite 
with parts $B_1$ and $B_2$ such that $z_1,z_3 \in B_1$ and $z_2,z_4 \in B_2$.
The homomorphism $\theta$ can be extended to $G$ by for $i \in \{1,2\}$ assigning 
$\theta(v) = \theta(z_i)$ for all $v \in B_i\setminus\{z_{i+2}\}$.
Observe that $G'$ is a forced extension quadrangulation if and only if $G$ is a forced extension quadrangulation.
If $G'$ is not a forced extension quadrangulation then it has another homomorphism $\theta'$
as described in the statement of the lemma, and this homomorphism also extends to $G$.
Hence, we can assume that every contractible $4$-cycle in $G$ bounds a face.

Clearly, the claim of the lemma holds if $G=C_1=C_2$.  Hence, we have $C_1\neq C_2$.  Suppose now that $C_1$
and $C_2$ intersect.  Then $G$ is a forced extension quadrangulation.  Furthermore, there exists a $4$-face
$z_1z_2z_3z_4$ such that $z_1\in V(C_1)\cap V(C_2)$.  We apply induction to the graph obtained from $G$ by identifying $z_2$
with $z_4$ (this graph has no parallel edges, since every contractible $4$-cycle in $G$ bounds a face),
and extend the resulting homomorphism to $G$ by setting $\theta(z_2)=\theta(z_4)$.

Suppose now that $C_1$ and $C_2$ are vertex-disjoint.  Since every non-contractible cycle distinct from $C_1$ and $C_2$
has length greater than $k$, every vertex of $C_1$ has degree at least three.  If all vertices of $C_1$ have degree exactly three,
then let $C_1=v_1v_2\ldots v_k$ and for $1\le i\le k$, let $u_i$ be the neighbor of $v_i$ not belonging to $C_1$.
Observe that $u_1u_2\ldots u_k$ is a $k$-cycle, which must be equal to $C_2$.  If $C=x_1x_2\ldots x_k$, we can define the homomorphisms
$\theta$ and $\theta'$ by setting $\theta(v_i)=\theta'(v_i)=x_i$, $\theta(u_i)=x_{i+1}$ and $\theta'(u_i)=x_{i-1}$ for $1\le i\le k$,
where $x_0=x_k$ and $x_{k+1}=x_1$.

Therefore, we can assume that a vertex $z_1$ of $C_1$ has degree at least $4$, and thus $G$ has a $4$-face $z_1z_2z_3z_4$ with $z_2,z_4\not\in V(C_1)$.
Since all non-contractible cycles in $G$ distinct from $C_1$ and $C_2$ have length at least $k+2$, at most one of $z_2$
and $z_4$ belongs to $C_2$, and the graph $G'$ obtained from $G$ by identifying $z_2$ with $z_4$ contains no non-contractible cycle of
length less than $k$.  Furthermore, since every contractible $4$-cycle in $G$ bounds a face, the graph $G'$ has no parallel edges.

Suppose that $G'$ contains a non-contractible $k$-cycle $K\neq C_1$ that intersects $C_1$ in a vertex $z$.
Then $G$ contains paths $P_1$ from $z_2$ to $z$ and $P_2$ from $z_4$ to $z$ such that $|P_1|+|P_2|=k$ and $K$ is obtained from $P_1\cup P_2$
by identifying $z_2$ with $z_4$.  Let $P'_1$ and $P'_2$ be the subpaths of $C_1$ between $z_1$ and $z$, chosen so that both closed walks
$K_1=P'_1\cup P_1\cup\{z_1z_2\}$
and $K_2=P'_2\cup P_2\cup\{z_1z_4\}$ are contractible.  Note that $|P'_1|+|P'_2|=k$, and thus by symmetry, we can assume that
$|P_1|\le |P'_1|$.  Since $K_1$ is contractible and all faces of $G$ other than the ones bounded by $C_1$ and $C_2$ have length $4$,
it follows that $|K_1|$ is even, and thus $P_1$ and $P'_1$ have opposite parity. Hence, $|P_1|\le |P'_1|-1$,
and $Q=P_1\cup P'_2\cup\{z_1z_2\}$ is a non-contractible closed walk of length at most $k$ in $G$.
Since $G$ contains no non-contractible cycles of length less than $k$, it follows that $Q$ is a cycle.
However, $Q$ is distinct from $C_1$ and $C_2$, which is a contradiction.

We conclude that $G'$ contains no such $k$-cycle $K$, and thus $G'$ is not a forced extension quadrangulation.
By the induction hypothesis, there exist homomorphisms $\theta$ and $\theta'$ from $G'$ to $C$ that are offset by $2$, and they can be extended
to homomorphisms from $G$ to $C$ satisfying the conditions of Lemma~\ref{lemma-hom} by setting $\theta(z_2)=\theta(z_4)$ and $\theta'(z_2)=\theta'(z_4)$.
\end{proof}

Next, let us recall a result from~\cite{trfree5}.

\begin{lemma}[\cite{trfree5}, Corollary 4.7]\label{lemma-joint}
Let $G$ be a quadrangulation of the cylinder with boundary cycles $C_1$ and $C_2$ of the same length $k$,
such that every non-contractible cycle in $G$ has length at least $k$, and the distance between $C_1$
and $C_2$ is at least $4k$.  Let $\psi$ be a precoloring of $C_1\cup C_2$ satisfying the winding number
constraint.  If $\psi$ does not extend to a $3$-coloring of $G$, then $k$ is divisible by $3$ and $\psi$ has winding number
$\pm k/3$ on $C_1$.
\end{lemma}

This makes it easy to analyze the colorings of wide forced extension quadrangulations.
Note that the existence of the homomorphism $\theta$ in the following statement is guaranteed by Lemma~\ref{lemma-hom}.

\begin{lemma}\label{lemma-feqcol}
Let $G$ be a wide forced extension quadrangulation of a cylinder with boundary cycles of length $k$,
and let $\theta:V(G)\to V(C)$ be a homomorphism to a $k$-cycle $C$
such that the restrictions of $\theta$ to both boundary cycles of $G$ are cyclic.
A precoloring $\psi$ of the boundary cycles of $G$ extends to a $3$-coloring of $G$ if and only if
\begin{itemize}
\item $\psi$ satisfies the winding number constraint, and additionally,
\item if $k$ is divisible by $3$ and $\psi$ has winding number $\pm k/3$ on the boundary cycles, 
then $\psi(u)=\psi(v)$ for all vertices $u$ and $v$ in the boundary cycles such that $\theta(u)=\theta(v)$.
\end{itemize}
\end{lemma}
\begin{proof}
Let $G$ be a wide forced extension quadrangulation with boundary cycles $C_1$ and $C_2$ of length $k$,
and let $K_1$, \ldots, $K_n$ be the non-contractible $k$-cycles such that $K_1=C_1$, $K_n=C_2$ and
$K_i$ intersects $K_{i+1}$ for $1\le i\le n-1$.

Let $\psi$ be any precoloring of $C_1\cup C_2$.  If $\psi$ does not satisfy the winding number constraint, then
it does not extend to a $3$-coloring of $G$, as we observed in the introduction (Corollary~\ref{cor-win0}).   Suppose that $\psi$ satisfies
the winding number constraint.  If the winding number of $\psi$ on $C_1$ is not $\pm k/3$, then $\psi$ extends to a $3$-coloring of $G$
by Lemma~\ref{lemma-joint}.

Finally, consider the case that the winding number of $\psi$ on $C_1$ is $\pm k/3$, and in particular, $k$ is divisible by $3$.
Let $C=v_1v_2\ldots v_k$.  Note that $\theta$ is cyclic on cycles $K_1$, \ldots, $K_n$: for each edge $e=uv$ of $G$ with $\theta(u)=v_a$
and $\theta(v)=v_b$, orient $e$ so that the head of $e$ is $v$ if and only if $b\equiv a+1\pmod k$, and for each walk $Q$ in $G$,
let $p(Q)$ denote the difference between the number of edges of $Q$ oriented forwards and backwards.  Since $k\neq 4$, we have $p(Q)=0$
for each $4$-cycle $Q$, and since all faces of $G$ other than the ones bounded by $C_1$ and $C_2$ have length $4$, we conclude that
$p(C_1)=p(K)$ for every non-contractible cycle $K$.  It follows that $p(K_1)=\ldots=p(K_n)=k$, and thus $\theta$ is cyclic
on these cycles.

For $1\le i\le n$, let $K_i=v_{i,1}v_{i,2}\ldots v_{i,k}$, where the labels are chosen so that $\theta(v_{i,j})=v_j$ for $1\le j\le k$.
Consider any $3$-coloring $\varphi$ of $G$, and any $i\in\{1,\ldots, n-1\}$.  Suppose that $\varphi$ has winding number $\pm k/3$
on $K_i$; without loss of generality, we have $\varphi(v_{i,j})=j\bmod 3$ for $1\le j\le k$.  By the winding number constraint,
$\varphi$ has the same winding number on $K_{i+1}$, and thus there exists $\alpha\in\{0,1,2\}$ such that
$\varphi(v_{i+1},j)=(j+\alpha)\bmod 3$ for $1\le j\le k$.  However, $K_i$ and $K_{i+1}$ share a vertex, and thus $\alpha=0$.
We conclude that if $\varphi$ has winding number $\pm k/3$ on $C_1$, then $\varphi(u)=\varphi(v)$ for any $u,v\in V(K_1)\cup \ldots\cup V(K_n)$
such that $\theta(u)=\theta(v)$.

Hence, in the case that $k$ is divisible by $3$ and $\psi$ has winding number $\pm k/3$ on the boundary cycles,
if $\psi$ extends to a $3$-coloring of $G$, then $\psi(u)=\psi(v)$ for all vertices $u$ and $v$ in the boundary cycles such that $\theta(u)=\theta(v)$.
Conversely, if $\psi$ satisfies this condition on the boundary cycles, then we can extend $\psi$ to a $3$-coloring $\varphi$ of $G$ by setting $\varphi(u)=\psi(u')$
for every $u\in V(G)$, where $u'$ is the vertex of $C_1$ such that $\theta(u)=\theta(u')$.
\end{proof}

\section{The structure theorem for quadrangulations of the cylinder}\label{sec-cyl}

Lemma 2.1 of Dvo\v{r}\'ak, Kr\'al' and Thomas~\cite{trfree6} gives the following result on precoloring extension in quadrangulations
of the disk.

\begin{lemma}\label{lemma-diskq}
Let $G$ be a quadrangulation of the disk with boundary cycle $C$.  Let $\psi$ be a $3$-coloring of $C$ with winding number $0$.
If $\psi$ does not extend to a $3$-coloring of $G$, then $G$ contains a path $P$ intersecting $C$ exactly in its endpoints
so that both cycles in $C\cup P$ distinct from $C$ are strictly shorter than $C$.
\end{lemma}

Iterating this lemma, we get the following result

\begin{corollary}\label{cor-disk}
Let $k\ge 4$ be an integer.  For every quadrangulation $G$ of the disk with the boundary cycle $C$ of length $k$,
there exists a subgraph $H\subseteq G$ with at most $2^{k/2}$ vertices,
such that $H$ contains the boundary cycle and for every face $h$ of $H$,
every precoloring of the boundary of $h$ which satisfies the winding number constraint extends to a $3$-coloring of $G[h]$.
\end{corollary}
\begin{proof}
We prove the claim by induction on $k$.  If $k=4$, then every precoloring of $C$ extends to a $3$-coloring of $G$ by the result of
Aksenov~\cite{aksenov}, and we can set $H=C$.  Hence, assume that $k\ge 6$.
If $G$ contains no path $P$ intersecting $C$ exactly in its endpoints
such that both cycles in $C\cup P$ distinct from $C$ are strictly shorter than $C$,
then by Lemma~\ref{lemma-diskq} we can again set $H=C$.

Hence, suppose that $G$ contains such a path $P$.   Let $C_1$ and $C_2$ be the cycles of $C\cup P$
distinct from $C$; we have $|C_1|,|C_2|\le k-2$.  For $i\in\{1,2\}$, let $G_i$ be the subgraph of $G$ drawn in the closed disk
bounded by $C_i$, and let $H_i$ be its subgraph obtained by the induction hypothesis.  We can set $H=H_1\cup H_2$;
note that $|V(H)|\le |V(H_1)|+|V(H_2)|\le 2\cdot 2^{(k-2)/2}=2^{k/2}$.
\end{proof}

We need another result of~\cite{trfree5}.
\begin{lemma}[\cite{trfree5}, Lemma 4.5]\label{lemma-cylcol}
Let $G$ be a quadrangulation of the cylinder with boundary cycles $C_1$ and $C_2$, such
that every non-contractible cycle in $G$ distinct from $C_1$ and $C_2$ has length greater than $\max(|C_1|,|C_2|)$.
If the distance between $C_1$ and $C_2$ is at least $|C_1|+|C_2|$, then a precoloring of $C_1\cup C_2$ extends
to a $3$-coloring of $G$ if and only if it satisfies the winding number constraint.
\end{lemma}

We also need a variation of Lemma~\ref{lemma-cylcol}.
\begin{lemma}\label{lemma-nonfeq}
Let $G$ be a quadrangulation of the cylinder with boundary cycles $C_1$ and $C_2$ of the same length $k$,
such that every non-contractible cycle in $G$ has length at least $k$.  Suppose that $G$ contains
a non-contractible $k$-cycle $K$, and for $i\in \{1,2\}$, let $G_i$ be the subgraph of $G$ drawn
between $K$ and $C_i$.  Assume that neither $G_1$ nor $G_2$ is a forced extension quadrangulation.
If the distance between $C_1$ and $C_2$ is at least $4k$, then a precoloring of $C_1\cup C_2$ extends
to a $3$-coloring of $G$ if and only if it satisfies the winding number constraint.
\end{lemma}
\begin{proof}
As we have shown in the introduction (Corollary~\ref{cor-win0}), the winding number constraint is a necessary condition for the precoloring to extend.
Let $\psi$ be a precoloring of $C_1\cup C_2$ that satisfies the winding number constraint.
If $\psi$ does not have winding number $\pm k/3$ on the boundary cycles, then the claim follows from Lemma~\ref{lemma-joint}.
Hence, assume that $k$ is divisible by $3$ and $\psi$ has winding number $\pm k/3$ on the boundary cycles.

Let $C$ be a $k$-cycle, $C=x_1x_2\ldots x_k$.
Let $\theta_1$ and $\theta'_1$ be the homomorphisms from $G_1$ to $C$ obtained by Lemma~\ref{lemma-hom}.
Let $C_1=v_1\ldots v_k$ and $K=w_1\ldots w_k$, with the labels chosen so that
$\theta_1(w_i)=\theta'_1(w_i)=\theta_1(v_i)=x_i$ and $\theta'_1(v_i)=x_{i+2}$ for $1\le i\le k$, where $x_{k+1}=x_1$ and $x_{k+2}=x_2$.
Let $\theta_2$ and $\theta'_2$ be the homomorphisms from $G_2$ to $C$ obtained by Lemma~\ref{lemma-hom}.
We can choose the homomorphisms and the labels of $C_2=z_1z_2\ldots z_k$ so that
$\theta_2(w_i)=\theta'_2(w_i)=\theta_2(z_i)=x_i$ and $\theta'_2(z_i)=x_{i+2}$ for $1\le i\le k$.

Let $\theta^\star_1=\theta_1\cup\theta_2$, $\theta^\star_2=\theta'_1\cup \theta_2$ and $\theta^\star_3=\theta_1\cup\theta'_2$.
For $i\in\{1,2,3\}$, let $\varphi_i$ be the $3$-coloring of $G$ defined by setting $\varphi_i(y)=\psi(v)$ for every $y\in V(G)$ and
$v\in V(C_1)$ such that $\theta^\star_i(y)=\theta^\star_i(v)$.  Observe that 
$\varphi_1\restriction V(C_1)=\varphi_2\restriction V(C_1)=\varphi_3\restriction V(C_1)=\psi\restriction V(C_1)$ and
that $\varphi_1\restriction V(C_2)\neq\varphi_2\restriction V(C_2)\neq \varphi_3\restriction V(C_2)\neq \varphi_1\restriction V(C_2)$.

Since $\psi$ has winding number $\pm k/3$, there are only three colorings of $C_2$ with the same winding number as $\psi$ has on $C_1$.
Consequently, one of $\varphi_1$, $\varphi_2$ and $\varphi_3$ extends $\psi$.
\end{proof}

We can now prove Theorem~\ref{thm-structmain} for quadrangulations of the cylinder, which exposes
the importance of the forced extension quadrangulations.

\begin{lemma}\label{lemma-feq}
For every $k\ge 3$, there exists a constant $\beta_2$ with the following property.  Let $G$ be a quandrangulation of the cylinder, such that
both boundary cycles of $G$ have length at most $k$.  Then $G$ has a subgraph $H$ with at most $\beta_2$
vertices, such that $H$ contains all the boundary cycles and each face $h$ of $H$ satisfies one of the following (where $G[h]$ is the subgraph
of $G$ drawn in the closure of $h$):
\begin{enumerate}[(a)]
\item every precoloring of the boundary of $h$ which satisfies the winding number constraint extends to a $3$-coloring of $G[h]$, or
\item $h$ is an open cylinder and $G[h]$ is its wide forced extension quadrangulation with boundary cycles whose length is divisible by $3$.
\end{enumerate}
\end{lemma}
\begin{proof}
Let $C_1$ and $C_2$ be the boundary cycles of $G$, where $|C_1|\le |C_2|=k$.
If $K$ is a shortest non-contractible cycle of $G$ distinct from $C_1$ and $C_2$, let $\ell(G)=\min(|K|,k+1)$.
We now proceed by induction on $k$, and subject to that on decreasing $\ell(G)$;
i.e., we assume that Lemma~\ref{lemma-feq} holds for all quadrangulations $G'$ with boundary cycles of length less than $k$,
and for those with boundary cycles of length at most $k$ and with $\ell(G')>\ell(G)$.

We will exhibit a sequence $K_1$, \ldots, $K_n$ of non-contractible $(\le\!k)$-cycles with $n\le 5$ such that
$C_1=K_1$, $C_2=K_n$, for $1\le i < j <m\le n$, the cycle $K_j$ separates $K_i$ from $K_m$, and
such that for $1\le i\le n-1$, the subgraph $G_i$ of $G$ drawn between $K_i$ and $K_{i+1}$ satisfies one of
the following properties:
\begin{enumerate}[(i)]
\item $\max(|K_i|, |K_{i+1}|)<k$, or
\item $\ell(G_i)>\ell(G)$, or
\item $\ell(G_i)=k+1$, or
\item $|K_i|=|K_{i+1}|=\ell(G_i)=k$ and $G_i$ contains a non-contractible $k$-cycle $K'_i$ such that neither
the subgraph of $G$ drawn between $K_i$ and $K'_i$ nor the subgraph drawn between $K'_i$ and $K_{i+1}$ is
a forced extension quadrangulation, or
\item $G_i$ is a forced extension quadrangulation.
\end{enumerate}

Before we construct such a sequence, we show how its existence implies Lemma~\ref{lemma-feq}.
If the distance between $K_i$ and $K_{i+1}$ is less than $4k$, then let $G'_i$ be the graph obtained from $G_i$
by cutting the cylinder along a shortest path $P_i$ between $K_i$ and $K_{i+1}$.  Note that $G'_i$ is embedded in the disk;
let $H'_i$ be the subgraph of $G'_i$ obtained by Corollary~\ref{cor-disk}, and let $H_i\subseteq G_i$ be obtained from $H'_i$
by merging back the vertices of the path $P_i$.

Otherwise, in cases (i) and (ii), let $H_i$ be the subgraph of $G_i$ obtained by the induction hypothesis.
In the cases (iii) and (iv), note that by Lemmas~\ref{lemma-cylcol} and \ref{lemma-nonfeq}, every  precoloring of $K_i\cup K_{i+1}$ which satisfies the winding number
constraint extends to a $3$-coloring of $G_i$, and set $H_i=K_i\cup K_{i+1}$.  In the case (v), since $G_i$ is a wide forced extension quadrangulation,
Lemma~\ref{lemma-feqcol} implies that either every precoloring of $K_i\cup K_{i+1}$ which satisfies the winding number
constraint extends to a $3$-coloring of $G_i$, or $|K_i|$ is divisible by $3$; we set $H_i=K_i\cup K_{i+1}$.

Observe that every face of $H=\bigcup_{i=1}^{n-1} H_i$ satisfies either the condition (a) or (b) of the statement of the lemma.
Hence, it suffices to construct the sequence $K_1$, \ldots, $K_n$ with the described properties.

If $\ell(G)=k+1$, then we set $n=2$, $K_1=C_1$ and $K_2=C_2$, and note that $G$ satisfies (iii).
If $\ell(G)<k$, then let $n=4$, let $K_1=C_1$, $K_4=C_2$
and let $K_2$ and $K_3$ be non-contractible $\ell(G)$-cycles chosen so that $G_1$ and $G_3$ are minimal (with possibly $K_2=K_3$).
Consequently, $G_1$ and $G_3$ satisfy (ii) and $G_2$ satisfies (i).

Hence, assume that $\ell(G)=k$.  We set $K_1=C_1$ and let $K_2$ be a non-contractible $k$-cycle chosen so that $G_1$ is minimal
(where possibly $K_2=C_1$ if $|C_1|=k$), so that
$G_1$ satisfies (ii).  If the subgraph of $G$ drawn between $K_2$ and $C_2$ satisfies (iii), (iv) or (v), then we set $n=3$ and $K_3=C_2$.
Otherwise, let $n=5$, $K_5=C_2$, and let $K_3$ and $K_4$ be chosen as non-contractible $k$-cycles such that $G_2$ and $G_4$
are forced extension quadrangulations and they are maximal with this property.  Consequently, every non-contractible
$k$-cycle in $G_3$ distinct from $K_3$ and $K_4$ is disjoint from $K_3\cup K_4$, and since the subgraph of $G$ drawn between $K_2$ and $C_2$
does not satisfy (iv), we conclude that $K_3$ and $K_4$ are the only non-contractible $k$-cycles in $G_3$, and thus $G_3$ satisfies (iii).
\end{proof}

We can now prove Theorem~\ref{thm-structmain} in the special case of graphs embedded in the disk.

\begin{theorem}\label{thm-structdisk}
For every integer $k\ge 4$, there exists a constant $\beta_3$ with the following property.  Let $G$ be a triangle-free graph embedded in
the disk whose cuff traces a cycle $C$ in $G$ of length at most $k$.
Suppose that every $4$-cycle in $G$ bounds a face.  Then $G$ has a subgraph $H$ with at most $\beta_3$
vertices, such that $C\subseteq H$ and each face $h$ of $H$ satisfies one of the following (where $G[h]$ is the subgraph
of $G$ drawn in the closure of $h$).
\begin{enumerate}[(a)]
\item Every precoloring of the boundary of $h$ extends to a $3$-coloring of $G[h]$, or
\item $G[h]$ is a quadrangulation and every precoloring of the boundary of $h$ which satisfies the winding number constraint
extends to a $3$-coloring of $G[h]$, or
\item $h$ is an open cylinder and $G[h]$ is its wide forced extension quadrangulation with boundary cycles whose length is divisible by $3$.
\end{enumerate}
\end{theorem}
\begin{proof}
Let $H_0$ be the subgraph of $G$ obtained by Theorem~\ref{thm-struct}.
If a face $h$ of $H_0$ satisfies (a) or (b), then let $F_h$ be the subgraph of $G$
consisting of the boundary walks of $h$.  If a face $h$ of $H_0$ satisfies (c), then
let $F_h$ be the subgraph of $G[h]$ obtained by Lemma~\ref{lemma-feq}.  Note that no face of $H_0$
satisfies (d), since $G$ does not contain non-facial $4$-cycles.
Observe that every face of the graph $H=\bigcup_h F_h$ satisfies one of the conditions (a), (b), or (c) of Theorem~\ref{thm-structdisk}.
\end{proof}

\section{The structure theorem}\label{sec-main}

In order to prove the structure theorem, we use a result that we derived in the second part of the series~\cite{cylgen-part2}.
We need a few more definitions.

Let $G$ be a graph embedded in the cylinder, with the boundary cycles $C_i=x_iy_iz_iw_i$
of length $4$, for $i=1,2$.  Let $y'_i$ be either a new vertex or $y_i$, and let $w'_i$ be either a new vertex
or $w_i$.  Let $G'$ be obtained from $G$ by adding $4$-cycles $x_iy'_iz_iw'_i$ forming the boundary cycles.  We say that
$G'$ is obtained by \emph{framing on pairs $x_1z_1$ and $x_2z_2$}.

Let $G$ be a graph embedded in the cylinder with boundary cycles $C_1$ and $C_2$ of length $3$, such that every face of $G$ has length $4$.
We say that such a graph $G$ is a \emph{$3,3$-quadrangulation}.
Let $G'$ be obtained from $G$ by subdividing at most one edge in each of $C_1$ and $C_2$.  We say that such a graph $G'$ is a \emph{near $3,3$-quadrangulation}.

If $G$ is a graph embedded in a surface and $C$ is the union of the boundary cycles of $G$, we say that $G$ is \emph{critical} if
$G\neq C$ and for every $G'\subsetneq G$ such that $C\subseteq G'$, there exists a $3$-coloring of $C$ that extends to a $3$-coloring
of $G'$, but does not extend to a $3$-coloring of $G$.

\begin{theorem}[Dvo\v{r}\'ak and Lidick\'y~\cite{cylgen-part2}]\label{thm-mfar}
There exists a constant $D\ge 0$ such that the following holds.
Let $G$ be a triangle-free graph embedded in the cylinder with boundary cycles $C_1$ and $C_2$ of length $4$.
If the distance between $C_1$ and $C_2$ is at least $D$ and $G$ is critical, then either
\begin{itemize}
\item $G$ is obtained from a patched Thomas-Walls graph by framing on its interface pairs, or
\item $G$ is a near $3,3$-quadrangulation.
\end{itemize}
\end{theorem}

Finally, we are ready to combine all the results.

\begin{proof}[Proof of Theorem~\ref{thm-structmain}]
Let $D$ be the constant of Theorem~\ref{thm-mfar}.
Let $H_0$ be the subgraph of $G$ obtained by Theorem~\ref{thm-struct}.  If a face $h$ of $H_0$ satisfies (a) or (b), then let $F_h$ be the subgraph of $G$
consisting of the boundary walks of $h$.  If a face $h$ of $H_0$ satisfies (c), then let $F_h$ be the subgraph of $G[h]$ obtained by Lemma~\ref{lemma-feq}.

Suppose that $h$ is a face of $H_0$ satisfying (d).  Let $C_1$ and $C_2$ be the boundary cycles of $G[h]$.
If the distance between $C_1$ and $C_2$ is less than $D$, then cut the cylinder in that $G[h]$ is embedded along
the shortest path $P$ between $C_1$ and $C_2$, let $F'_h$ be the graph obtained by applying Theorem~\ref{thm-structdisk}
to the resulting graph embedded in the disk, and let $F_h$ be obtained from $F'_h$ by gluing back the two paths
created by cutting. 

Assume now that the distance between $C_1$ and $C_2$ in $G[h]$ is at least $D$.  If every $3$-coloring of $C_1\cup C_2$ extends to a
$3$-coloring of $G[h]$, then let $F_h=C_1\cup C_2$.  Otherwise, let $F'$ be a maximal critical subgraph of $G[h]$,
and apply Theorem~\ref{thm-mfar} to $F'$.
\begin{itemize}
\item If $F'$ is obtained from a patched Thomas-Walls graph by framing on its interface pairs,
then note that all faces of $G'$ have length at most $5$, and thus by the assumptions of Theorem~\ref{thm-structmain},
we have $F'=G[h]$.  Let $F_h$ consist of the union of the $4$-cycles containing an interface pair of $F'$ (hence, $|V(F_h)|\le 12$).
\item If $F'$ is a near $3,3$-quadrangulation, then similarly to the previous case, we have $F'=G[h]$.
Let $C'_1$ and $C'_2$ be the $5$-cycles in $F'$ such that for $i=1,2$, $C_i\cap C'_i$ is a path of length two
and the symmetric difference of $C_i$ and $C'_i$ bounds a $5$-face.  Let $F_h$ consist of $C_1\cup C'_1\cup C_2\cup C'_2$
and the subgraph of the quadrangulation between $C'_1$ and $C'_2$ obtained by Lemma~\ref{lemma-feq}.
\end{itemize}

Observe that every face of the graph $H=\bigcup_h F_h$ satisfies one of the conditions (a), (b), (c) or (d) of Theorem~\ref{thm-structmain}.
\end{proof}

\bibliographystyle{acm}
\bibliography{cylgen}

\end{document}